\documentclass[11pt,letterpaper]{amsart}

\newtheorem{theorem}{Theorem}[section]
\newtheorem{lemma}[theorem]{Lemma}

\theoremstyle{definition}
\newtheorem{definition}[theorem]{Definition}

\newtheorem{proposition}[theorem]{Proposition}

\theoremstyle{remark}
\newtheorem{remark}[theorem]{Remark}

\numberwithin{equation}{section}

 \newcommand\dd{\,\mathrm{d}}

\usepackage[colorlinks,linkcolor=cyan,citecolor=black]{hyperref}

\theoremstyle{plain}

\begin{document}

\title[A continuity method for the Dirichlet problem]{On a continuity method for Dirichlet problem of Hessian equations}


 \author{Rirong Yuan}
 \address{School of Mathematics, South China University of Technology, Guangzhou 510641, China}
 \email{yuanrr@scut.edu.cn}



\date{}

\dedicatory{}

\commby{}

\begin{abstract}
   
   In this paper, we  develop a continuity method for the Dirichlet problem of Hessian equations on Riemannian manifolds. Such equations, introduced by Caffarelli, Nirenberg and Spruck, are defined in terms of the eigenvalues of the Hessian and a given pair 
   $(f,\Gamma)$, where 
    $f$ is a symmetric function defined in a symmetric cone 
$\Gamma\subset \mathbb{R}^n$, and 
  $\Gamma$ specifies the set of admissible eigenvalues for the solution. Our method combines techniques from Morse theory with a characterization of the pair  $(f,\Gamma)$. More precisely, in the type 2 case, we first construct admissible functions using Morse theory, and then solve the Dirichlet problem without any additional assumptions on the boundary or the subsolution. Building on this characterization of the pair, we can approximate the type 1 equation by a family of type 2 equations.
\end{abstract}

\maketitle



  \section{Introduction}

 Let $(\overline{M}, g)$ be a $n$-dimensional  compact connected  Riemannian manifold   with smooth boundary $\partial M$. Denote $\overline{M}:=M\cup\partial M$, where  $n\geq 2$ and $M$ is the interior of $\overline{M}$. Let $\chi$ be a smooth symmetric $(0,2)$-tensor on $\overline{M}$, let    
 $\psi\in C^\infty(\overline{M})$ and 
 $\varphi\in C^{\infty}(\partial M)$.      
 Moreover, denote $\vec{\bf1}=(1,1,\cdots,1)\in\mathbb{R}^n.$
 
 In this paper, we propose a  continuity method for the Dirichlet problem 
 \begin{equation}
 	\label{equ1-real}
 	\begin{aligned}
 		F(\nabla^2u+\chi)
 		=
 		\psi \textrm{ in } \overline{M}, 
 	\end{aligned}
 \end{equation}
 \begin{equation}
 	\label{boundary-data1}
 	\begin{aligned}
 		u=\varphi \textrm{ on } \partial M,
 	\end{aligned}
 \end{equation}
 where $\nabla^2u$ is the real Hessian of $u$ with respect to the Levi-Civita connection $\nabla$ of $g$. 
 The  nonlinear operator 
 $F(\nabla^2u+\chi)$   is generated by a pair  $(f,\Gamma)$  and   has the form 
 \begin{equation}
 	\begin{aligned}
 		F(\nabla^2u+\chi)
 		= 	f(\lambda[\nabla^2u+\chi]), \nonumber
 	\end{aligned}
 \end{equation}
 where $\lambda[\nabla^2u+\chi]$  is an $n$-tuple of 
 eigenvalues of $\nabla^2u+\chi$ with  respect to $g$, and 
 $f$  is a smooth symmetric function of $n$ real variables, defined in an open symmetric  convex cone $\Gamma$  with vertex at the origin, and
 $$\Gamma_n:= \{(\lambda_1,\cdots,\lambda_n)\in \mathbb{R}^n: \mbox{ each } \lambda_i>0\}\subseteq\Gamma,$$
 as well as
 $\partial \Gamma\neq\emptyset$.  
 The function   
 $f$ is supposed to satisfy 
 \begin{equation}
 	\label{concave}  
 	\begin{aligned} 
 		f \mbox{ is a concave function in } \Gamma,	
 	\end{aligned}
 \end{equation} 
 \begin{equation}
 	\label{elliptic}
 	\begin{aligned}
 		f_{i}(\lambda)
 		=  f_{\lambda_i}(\lambda) :=\frac{\partial f}{\partial \lambda_{i}}(\lambda)
 		> 0  \mbox{ in } \Gamma, \mbox{  } \forall  1\leq  i\leq  n.
 	\end{aligned}
 \end{equation}
 
 \begin{definition}
 	\label{def1-admissiblefunction-complex}
 	For \eqref{equ1-real}, following \cite{CNS3} we say that $u\in C^{2}(\overline{M})$ is an admissible function if 
 	$\lambda[\nabla^2u+\chi]\in \Gamma$.

 \end{definition}
 
 The equation \eqref{equ1-real} includes the     Monge-Amp\`ere equation, 
 and more general $k$-Hessian equation
 corresponding to
 \begin{equation}	f=\sigma_k^{1/k}, \,\, \Gamma=\Gamma_k, \nonumber\end{equation} 
 where  $\sigma_k$ is the $k$-th elementary symmetric function,
 and $\Gamma_k$ is the  $k$-th G{\aa}rding cone. 
 Recently, an example introduced by \cite{Guan2021Zhang} has attracted considerable interest: 
 \begin{equation}	\label{Guan-Zhang-3}	
 	\begin{aligned} 
 		f= \frac{\sigma_{k}}{\sigma_{k-1}}		 	-\sum_{j=0}^{k-2} \frac{\beta_j \sigma_j}{\sigma_{k-1}},  \,\, \forall	 \beta_j\geq0,\,  	  \sum_{j=0}^{k-2}  \beta_j >0,
 		\,  \Gamma=\Gamma_{k-1}.
 	\end{aligned}
 \end{equation} 
 Moreover, the equation \eqref{equ1-real} also includes the following as a special case 
 \begin{equation}
 	\label{equ2-n-1}
 	\begin{aligned}
 		{f}(\lambda[ \Delta u \cdot g-  \nabla^2u +\chi])
 		=\psi. 
 	\end{aligned}
 \end{equation}
 When $ {f}(\lambda)=\sum_{i=1}^n \log\lambda_i$,
 it is a Monge-Amp\`ere equation for $(n-1)$-PSH functions in the sense of   \cite{Harvey2011Lawson}. 
 Such a type of equation is of special interest in complex geometry, since its close connection to form-type Calabi-Yau equation \cite{FuWangWuFormtype2010} and Gauduchon's conjecture \cite{Popovici2015,Tosatti2019Weinkove}.

 When $\chi=0$ and $M=\Omega\subset\mathbb{R}^n$, the Dirichlet problem  was initially treated 
 by Caffarelli-Nirenberg-Spruck \cite{CNS3}  under  some additional assumptions 
 on    $f$, including  
 \begin{equation}
 	\label{addistruc}
 	\begin{aligned} 
 		\lim_{t\rightarrow +\infty}f(t\lambda)=+\infty, 
 		\mbox{ for any } \lambda \in \Gamma,
 	\end{aligned}
 \end{equation} 
 and the unbounded condition
 \begin{equation}
 	\label{unbounded-1}
 	\begin{aligned}
 		\lim_{t\rightarrow+\infty}  f(\lambda_1,\cdots,\lambda_{n-1},\lambda_n+t)=+\infty, \,\, \forall\lambda=(\lambda_1,\cdots,\lambda_n)\in\Gamma.
 	\end{aligned}
 \end{equation}   
 Similar results were also obtained in literature  \cite{Krylov83,Ivochkina83}. 
 Caffarelli-Nirenberg-Spruck's result
 has been further extended by Trudinger \cite{Trudinger95}  to the bounded case. 
 In addition, their works requires  certain assumption on the principal curvatures of $\partial\Omega$. 
 Without geometric restriction to boundary,   as developed by \cite{Guan1994The,Guan1993Spruck} and  others, 
 in general the 
 existence result for Dirichlet problem requires  
 an appropriate   subsolution assumption:

 \begin{definition} 
 	[Subsolution]
 	We say that $\underline{u}\in C^2(\overline{M})$ is an admissible subsolution to the Dirichlet problem \eqref{equ1-real}-\eqref{boundary-data1}, if it is an admissible function and satisfies 
 	\begin{equation}
 		\label{subsolution1-complex}
 		\begin{aligned}
 			f(\lambda[ \nabla^2\underline{u} +\chi]) \geq \psi \textrm{ in } \overline{M}, \,\,
 			\underline{u}=\varphi \textrm{ on } \partial M. \nonumber
 		\end{aligned}
 	\end{equation}
 \end{definition}

 Condition \eqref{elliptic} ensures the equation \eqref{equ1-real} to be elliptic at any admissible solutions,
 while \eqref{concave} implies that the operator $F(A)=f(\lambda(A))$ is concave with respect to $A$ when $\lambda(A)\in\Gamma$.
 Therefore, under the assumption  
 \begin{equation}
 	\label{nondegeneracy1}
 	\begin{aligned}
 		\inf_M	\psi   > \sup_{\partial\Gamma} f, 
 		\textrm{ where } \underset{\partial \Gamma}\sup f =		\underset{\lambda_{0}\in \partial \Gamma}\sup \underset{\lambda\rightarrow \lambda_{0}}\limsup f(\lambda),
 	\end{aligned}
 \end{equation}
 once the following  a priori estimate  of  the  Hessian 
 is in place,   
 \begin{equation}
 	\label{estimate1-Hessian-real}
 	\begin{aligned}
 		\sup_{\overline{M}}	|\nabla^2 u|\leq C, 
 	\end{aligned}
 \end{equation}
 the higher regularities follow from the Evans-Krylov theorem and  Schauder theory.  
 Building on this, one can use standard continuity method to derive the existence. 
 
 This paper is devoted to developing a continuity method for the Dirichlet problem. 
 First we review some notation.
 For a given pair
 $(f,\Gamma)$, as in \cite{yuan1-closed}  
 denote 
 \begin{equation}
 	\label{component1}
 	\begin{aligned}
 		\Gamma_{\mathcal{G}}^{f} =
 		\big\{\lambda\in\Gamma: \lim_{t\to +\infty}f(t\lambda)>-\infty \big\}. \nonumber
 	\end{aligned}
 \end{equation}
 In addition, 
 $\mathring{\Gamma}_{\mathcal{G}}^{f}$
 denotes the interior of $\Gamma_{\mathcal{G}}^{f}$.
 For  
 ${\Gamma}_{\mathcal{G}}^{f}$, 
 as in  \cite{yuan-PUE2-note,yuan-PUE4} define
 \begin{equation}
 	\label{def2-varrho-Gamma}
 	\begin{aligned}
 		(1,\cdots,1, 1-\varrho_{{\Gamma}_{\mathcal{G}}^f}) \in    \partial\Gamma_{\mathcal{G}}^{f}.
 	\end{aligned}
 \end{equation}  
 Following  the convention in \cite{CNS3}, 
 as in
  \cite{yuan1-closed,yuan-PUE4} we say that  $(f,\Gamma)$ is a type 2 pair if 
 \begin{equation}
 	\label{assumption1-type2}
 	\begin{aligned}
 		(0,\cdots,0,1)\in\mathring{\Gamma}_{\mathcal{G}}^{f}.
 	\end{aligned}
 \end{equation}
 Otherwise, it is a type 1 pair. Accordingly, we call \eqref{equ1-real} a type 2 equation if the corresponding pair $(f,\Gamma)$ satisfies \eqref{assumption1-type2}. Otherwise it is a type 1 equation.

 \begin{remark}
 	
 	When $f$ satisfies \eqref{addistruc}, we know ${\Gamma}_{\mathcal{G}}^{f}=\Gamma$.
 	 	Notably, $\mathring{\Gamma}_{\mathcal{G}}^{f}\neq\Gamma_n \iff \varrho_{{\Gamma}_{\mathcal{G}}^f}>1.$
 	For \eqref{Guan-Zhang-3}, we have $\mathring{\Gamma}_{\mathcal{G}}^{f}=\Gamma_k$ and $\varrho_{{\Gamma}_{\mathcal{G}}^f}=\frac{n}{k}$. 

 \end{remark}
 
 The type 2 pair is much more simpler, as detailed in Lemma \ref{lemma1-PUE}. This is a  special case of the partial uniform ellipticity developed by the author in a series of works \cite{yuan2020conformal,yuan-PUE-conformal,yuan-PUE4}. 
 In this type 2 case, 
 together with a Morse theory technique that was first introduced by \cite{yuan-PUE-conformal} as a tool to construct admissible functions,  we obtain the existence result for the Dirichlet problem   without extra assumptions on the boundary or the subsolution.

 \begin{theorem}
 	\label{thm1-type2}
 	In addition to \eqref{concave}, \eqref{elliptic}  and \eqref{nondegeneracy1},
 	we assume 
 	that $(f,\Gamma)$ is of type 2.
 	Then  the 
 	Dirichlet problem \eqref{equ1-real}-\eqref{boundary-data1} has a unique smooth admissible solution.  
 \end{theorem}

 This is a fully nonlinear analogue of 
 existence theorem for Poisson's equation. 
 Significantly,   in the type 2 case  we can solve the Dirichlet problem 
 without any extra assumption regarding   admissible
 subsolution  or admissible function. 
 This stands in contrast with vast existing literature.

 On the other hand, 
 using the constant $\varrho_{{\Gamma}_{\mathcal{G}}^f}$, 
 we can give a characterization of $(f,\Gamma)$ in 
 Lemmas \ref{proposition2-Gamma} and \ref{lemma1-PUE}. 
 Thus we can recast   \eqref{equ1-real} 
 in the following form:
 \begin{equation}
 	\label{equ2-n-varrho}
 	\begin{aligned}
 		\tilde{f}(\lambda[ \Delta u \cdot g-\varrho \nabla^2u+\tilde{\chi}])
 		=\psi,  
 	\end{aligned}
 \end{equation}
 corresponding to $(\tilde{f},\tilde{\Gamma})$, 
 and vice versa,
 where  $$ \varrho\leq \varrho_{\tilde{\Gamma}_{\mathcal{G}}^{\tilde{f}}},  \varrho\neq0 
 \textrm{ and } \tilde{\chi}=(\mathrm{tr}_{g}\chi)\cdot g-\varrho \chi.$$ 
 In the critical case $\varrho=1$,  
 it reads as   \eqref{equ2-n-1}.
 Significantly,  
 when $\varrho<\varrho_{\tilde{\Gamma}_{\mathcal{G}}^{\tilde{f}}}$ and $ \varrho\neq0$ 
 the equation
 \eqref{equ2-n-varrho} 
 can be reformulated as
 an  equation \eqref{equ1-real}  of type 2, and vice versa. 
 See Proposition \ref{thm2-equivalent-equations}.
 Motivated by 
 this, in  Theorem \ref{thm1-approx} we  propose an approximation of a pair $(f,\Gamma)$ via type 2 pairs $(f^t,\Gamma^t)$.  
 An analogous approximation was also considered in \cite{Gursky-2003jdg-Viaclovsky,Li2003YYLi}.
 Our strategy is different and new. It is based on the  applications of Morse theory
 to fully nonlinear equations,
  which was 
  first observed by the author \cite{yuan-PUE-conformal}  
 (see also \cite{yuan-PUE2-note,yuan-PUE4,yuan-PUE3,yuan-construction4,yuan5-construction}).

 By integrating Theorems \ref{thm1-type2} and 
 \ref{thm1-approx}, 
 we propose  a 
 continuity method for the Dirichlet problem and apply it to establish a proof of the following result. This is a real analogue of \cite[Theorem 1.6]{yuan1-closed}.
 \begin{theorem}
 	\label{thm1-exietnce-Dirichlet-real}
 	Let $f$ satisfy \eqref{concave}, \eqref{elliptic}, \eqref{nondegeneracy1} and    \eqref{unbounded-1}.
 	Suppose that there exists an admissible subsolution $\underline{u}\in C^2(\overline{M})$ to  Dirichlet problem 
 	\eqref{equ1-real}-\eqref{boundary-data1}. 
 	Then the Dirichlet problem 
 	admits a unique smooth admissible solution.
 \end{theorem}

 In contrast to the usual approaches, 
 our method leverages the a priori estimates directly, eliminating the necessity of smooth approximations to the $C^2$-regularity admissible subsolution.
 In the unbounded-condition case, a similar approach also works on closed manifolds.

 Definitely not too surprisingly, things become  more subtle in the degenerate case 
 \begin{equation}
 	\label{de-RHS}
 	\begin{aligned} 
 		\inf_M \psi=\sup_{\partial\Gamma}f.
 	\end{aligned} 
 \end{equation}
 The optimal regularity one can expect  is  $C^{1,1}$. For a type 2 degenerate equation, we derive $C^{1,1}$-weak solution to the Dirichlet problem.
 To this end, we further assume 
 \begin{equation}
 	\label{continuity1}
 	\begin{aligned}
 		f\in C^\infty(\Gamma)\cap C(\overline{\Gamma}), \mbox{ where } \overline{\Gamma} =\Gamma\cup\partial\Gamma.
 	\end{aligned}
 \end{equation}

 \begin{theorem}
 	\label{thm3-type2-degenerate}
 	Assume 	\eqref{concave}, \eqref{elliptic}, \eqref{assumption1-type2},  \eqref{de-RHS} and \eqref{continuity1} hold. Then
 	the Dirichlet problem
 	\eqref{equ1-real}-\eqref{boundary-data1} possesses  a weak solution  
 	$u\in C^{1,1}(\overline{M})$ with 
 	$\Delta u \in L^{\infty}(\overline{M})$ and  $\lambda(\nabla^2 u+\chi)\in \overline{\Gamma}$ in $\overline{M}$.
 \end{theorem}

 \begin{remark}
 	
 	Theorems \ref{thm1-type2} and \ref{thm3-type2-degenerate} allow
 	\eqref{equ2-n-1} when $\mathring{ {\Gamma}}_{\mathcal{G}}^{ {f}}\neq\Gamma_n$, 
 	in which case   
 	$\varrho_{ {\Gamma}_{\mathcal{G}}^{ {f}}}>1$.
 \end{remark}

 One of the key ingredients in deriving \eqref{estimate1-Hessian-real} is to derive  gradient estimate  
 \begin{equation}
 	\label{gradient-estimate1}
 	\begin{aligned} 
 		\sup_M |\nabla u|\leq C.
 	\end{aligned}
 \end{equation}
 For the equation \eqref{equ1-real} with $\chi=g$, the gradient estimate  has been studied by Li \cite{LiYY1990} when the sectional curvature of $g$ is nonnegative, and  by Urbas \cite{Urbas2002} when replacing such a curvature condition by certain extra assumptions on $f$ including 
 \begin{equation}
 	\label{key2-yuan}
 	\begin{aligned}
 		f_i(\lambda)\geq \delta \sum  f_j(\lambda) \mbox{ if } \lambda_i\leq 0.
 	\end{aligned}
 \end{equation}
 See also \cite{Guan12a}. 
 In 2020,  
 the author \cite{yuan2020conformal}
 confirmed such a key inequality
 when $f$ satisfies \eqref{addistruc},
 thereby  
 extending  \eqref{estimate1-Hessian-real} to a large number
 of
 real Hessian equations. 
 A straightforward proof of the gradient estimate for \eqref{equ1-real}  was carried out by Guan \cite{Guan-Dirichlet} under  a rather weak assumption.
 The above-mentioned  
 gradient estimate is based on Bernstein method. 
 Another effective method is the blow-up argument. 
 Such an argument  was 
 developed by Chen \cite{Chen} 
 for Dirichlet problem of complex Monge-Amp\`ere equation,  and by  Dinew-Ko{\l}odziej \cite{Dinew2017Kolo},  building on Hou-Ma-Wu's  \cite{HouMaWu2010} second order
  estimate od the form
  \begin{equation}
 	\label{hmw-1}
 	\begin{aligned}
 		\sup_{M}|\Delta u|\leq C(1+\sup_M|\nabla u|^2),
 	\end{aligned}
 \end{equation}
 in the case of  complex $k$-Hessian  equations on closed  (compact without boundary) K\"ahler manifolds.   
 Recently, under the extra assumption 
 \eqref{addistruc}, the blow-up argument was extended extensively by  Sz\'ekelyhidi
 \cite{Gabor}   to Hessian equations on closed Hermitian manifolds. 
  Such a technical assumption  
  was subsequently removed by 
 the author \cite{yuan1-closed}, thereby 
 extending Sz\'ekelyhidi's $C^{2,\alpha}$-estimate to the optimal setting. 
Moreover,  in the unbounded case \eqref{unbounded-1}, 
 the author \cite{yuan1-closed} also established  a prior 
  estimate   \eqref{hmw-1}   
 and then solved the Dirichlet problem 
 of Hessian equations on complex manifolds.  
 In this paper, we employ their strategy to derive gradient estimate 
 in the type 1 case.

 The rest of this paper is organized as follows.
 In Section \ref{sec2-approximate-pairs} we  discuss the structure of a pair and then propose an approximation of  the pair $(f,\Gamma)$ by type 2 pairs. 
 In Section \ref{sec3-proof-of-type2-existence}, we prove Theorems \ref{thm1-type2}
 and \ref{thm3-type2-degenerate}
 based on Morse theory and the characterization of type 2 pairs. 
 In Section \ref{sec4} we establish 
 \eqref{hmw-1} and \eqref{estimate1-Hessian-real}. 
 In Section \ref{sec2-contunuitymethod-Dirichlet} we propose a continuity method for  Dirichlet problem  and  employ it to give a proof of Theorem   \ref{thm1-exietnce-Dirichlet-real}. 
 In Appendix \ref{appendix1} we summarize some useful lemmas.

 \section{Approximate Hessian equations using type 2 equations}
 \label{sec2-approximate-pairs}
 
 
 \subsection{Structure of a type 2 pair} 
 
 We  review a property of $ \Gamma_{\mathcal{G}}^{f}$ proposed in \cite{yuan1-closed}.
 
 \begin{lemma}[\cite{yuan1-closed}]
 	\label{lemma2-key}
 	Suppose that $f$ satisfies  \eqref{concave}   in $\Gamma$. Then  
 	\begin{equation}
 		\label{inequality22} 
 		\begin{aligned}
 			\sum_{i=1}^n f_i(\lambda)\mu_i \geq  0  \textrm{ and } 
 			f(\lambda+\mu)\geq  f(\lambda), 
 			\,\, \forall \lambda\in\Gamma, \, 
 			\forall	\mu\in  \Gamma_{\mathcal{G}}^{f}.
 		\end{aligned}
 	\end{equation} 
 	
 \end{lemma}
 \begin{proof} 
 	We start with an observation  
 	\begin{equation}
 		\label{010} 
 		\begin{aligned}
 			\sum_{i=1}^n f_i(\lambda)\mu_i \geq \limsup_{t\rightarrow+\infty} f(t\mu)/t, 
 			\,\, \forall \lambda, \, \mu\in\Gamma.  
 	\end{aligned}\end{equation} 
 	This inequality follows from  the concavity of $f$. 
 	Moreover, by \eqref{concave}  we get
 	\begin{equation}
 		\label{concave-1}
 		\begin{aligned}
 			\sum_{i=1}^n f_i(\lambda)(\mu_i-\lambda_i)\geq f(\mu)-f(\lambda), \,\, \forall \lambda, \,
 			\mu \in \Gamma.
 		\end{aligned}
 	\end{equation}
 	By \eqref{010}, 
 	$\sum  f_i(\lambda)\mu_i \geq  0$, 
 	$\forall\lambda\in\Gamma$, $\forall\mu\in 	\Gamma_{\mathcal{G}}^{f}$.  
 	Combining \eqref{concave-1}, we get
 	$f(\lambda+\mu)\geq  f(\lambda)$.

 \end{proof}
 
 As in introduction  $\vec{\bf1}=(1,\cdots,1).$
 According to  \cite[Theorem 5.6]{yuan-PUE4},
 we get
 \begin{lemma}
  	[\cite{yuan-PUE4}]
 	\label{lemma1-PUE}
 	In addition to  \eqref{concave},
 	we assume $(-\delta_0,0,\cdots,0,1)\in \mathring{\Gamma}_{\mathcal{G}}^f$ for some $\delta_0>0$. Then 
 	\begin{equation}
 		\label{pue-3} 
 		\begin{aligned}
 			f_i(\lambda)\geq    \frac{\delta_0}{n}\sum_{j=1}^n f_j(\lambda) \textrm{ in } \Gamma, \,\,  \forall 1\leq i\leq n.  
 		\end{aligned}
 	\end{equation}
 	Moreover,  
 	$\underset{t\rightarrow+\infty} \lim f(\lambda_1,\cdots,\lambda_{n-1},\lambda_n+t)=\sup_\Gamma f, \,\, \forall \lambda=(\lambda_1,\cdots, \lambda_n)\in\Gamma.$   
 \end{lemma}
 
 \begin{remark} 
 	When $f$ further satisfies \eqref{addistruc}, 
 	the result was  obtained  in \cite{yuan2020conformal}.
 \end{remark}

 \begin{proof}
 	
 	Suppose $\lambda_1 \leq \cdots\leq \lambda_n$. The concavity and symmetry of $f$ implies 
 	$f_n\leq \cdots \leq f_1$ and  $f_1\geq \frac{1}{n}\sum  f_j$.   
 	By Lemma \ref{lemma2-key}, 
 	$f_n\geq \delta_0 f_1\geq \frac{\delta_0}{n}\sum f_j$.
 	Finally,   for  $R_0>0$  from   $(0,\cdots,0,1)\in \mathring{\Gamma}_{\mathcal{G}}^f$, 
 	one can find  $t\gg1$ so that 	$t(0,\cdots,0,1)+\lambda-R_0\vec{\bf1}\in \mathring{\Gamma}_{\mathcal{G}}^f$. Thus 
 	$	f(\lambda_1,\cdots,\lambda_{n-1},\lambda_n+t)
 	\geq f(R_0\vec{\bf1})$ 
 	according to \eqref{inequality22} from Lemma \ref{lemma2-key}.

 \end{proof}

 \subsection{Structure of a pair and an application to approximation}
 \label{subsect-2.2}
 The results of this subsection are based on an observation.
 The observation is noted in \cite{yuan-PUE4} and closely aligns with the results presented in \cite[Section 4]{yuan-PUE3}.
 Fix a   pair  $({f},{\Gamma})$.
 Let $\varrho_{{\Gamma}_{\mathcal{G}}^f}$ be as in \eqref{def2-varrho-Gamma}, 
 and pick a positive constant $\varrho$ with
 $n-\varrho_{{\Gamma}_{\mathcal{G}}^f} < \varrho <n$.
 Define
 $\mu^{\lambda}	:= \sum_{j=1}^n \lambda_j \vec{\bf 1}-\varrho\lambda, \, \forall \lambda\in\mathbb{R}^n.$   
 Let $P$ be a linear map   from $\mathbb{R}^n$  to $\mathbb{R}^n$ 
 with $P(\lambda)= \mu^\lambda.$ 
 Since  $ 0<\varrho <n$,  the linear map   is well-defined and invertible, and $P^{-1}(\mu)=\frac{1}{\varrho(n-\varrho)}(\sum_{j=1}^n \mu_j \vec{\bf1}-(n-\varrho)\mu)$.  
 Take $\tilde{\Gamma}:=P({\Gamma})$ and define  $ \tilde{f}: \tilde{\Gamma} \to \mathbb{R}$ by
 \begin{equation}
 	\label{def1-f}\begin{aligned}  
 		\tilde{f}(\mu^\lambda)= {f}(\lambda). 
 \end{aligned}\end{equation} 
 So    $\tilde{\Gamma}$ is an open, convex, symmetric cone with vertex at origin.
 We prove the following lemma,   which is noted in  \cite{yuan-PUE4}.
 \begin{lemma}
 	\label{proposition2-Gamma}
 Given a pair $({f},{\Gamma})$ obeying \eqref{concave} and \eqref{elliptic}. Let $\varrho$, $\tilde{\Gamma}$ and $\tilde{f}$  be as above. 
 	Then   $\tilde{\Gamma}$ is an open symmetric   convex cone with vertex at the origin, and 
 	$$ 
 	\Gamma_n\subseteq \mathring{\tilde{\Gamma}}_{\mathcal{G}}^{\tilde{f}}\subseteq \tilde{\Gamma}\subseteq \Gamma_1, \,\, \varrho\leq \varrho_{\tilde{\Gamma}_{\mathcal{G}}^{\tilde{f}}},$$
 	and $\tilde{f}$ satisfies \eqref{concave} and  \eqref{elliptic} in $\tilde{\Gamma}$.
 	Moreover, $(\tilde{f},\tilde{\Gamma})$ is of type 2, 
 	and 
 	\begin{enumerate}
 		\item If $(0,\cdots,0,1)\in\mathring{{\Gamma}}_{\mathcal{G}}^{{f}}$, then $\varrho<\varrho_{\tilde{\Gamma}_{\mathcal{G}}^{\tilde{f}}}.$
 		
 		\item If $(0,\cdots,0,1)\in\partial{{\Gamma}}_{\mathcal{G}}^{{f}}$, then $\varrho=\varrho_{\tilde{\Gamma}_{\mathcal{G}}^{\tilde{f}}}.$
 	\end{enumerate}
 	
 \end{lemma}
 
 \begin{proof}
 	From $\sum_{j=1}^n \mu_j^{\lambda}=(n-\varrho) \sum_{j=1}^n\lambda_j$, we see $\tilde{\Gamma}\subseteq\Gamma_1.$   
 	By $ n-\varrho_{{\Gamma}_{\mathcal{G}}^f}< \varrho$, we know $(1,\cdots,1,1-n+\varrho)\in  \mathring{\Gamma}_{\mathcal{G}}^f,$ thereby $(0,\cdots,0,1)\in \mathring{\tilde{\Gamma}}_{\mathcal{G}}^{\tilde{f}}$. 
 	By	${\tilde{\Gamma}}_{\mathcal{G}}^{\tilde{f}}=P({\Gamma}_{\mathcal{G}}^f)$ and  $(0,\cdots,0,1)\in\overline{{\Gamma}}_{\mathcal{G}}^{{f}}$, we get $(1,\cdots,1,1-\varrho)\in\overline{\tilde{\Gamma}}_{\mathcal{G}}^{\tilde{f}}$. Thus $\varrho\leq \varrho_{\tilde{\Gamma}_{\mathcal{G}}^{\tilde{f}}}.$ 
 	On the other hand,    
 	we can  check the concavity of $\tilde{f} $ in $\tilde{\Gamma}$,  thanks to 
 	\begin{gather*} 	 
  \tilde{f}(t\mu^{\lambda}+(1-t)\mu^{\lambda'})
 		= f(t\lambda+(1-t)\lambda')   
 			\geq   tf(\lambda)+(1-t)f(\lambda')
 			\\
 		=t\tilde{f}(\mu^{\lambda})
 		+ (1-t)\tilde{f}(\mu^{\lambda'}),
 	 \,\, \forall 0<t<1, \, \forall \mu^{\lambda},\,\mu^{\lambda'}\in \tilde{\Gamma}.
 	\end{gather*} 
 	Finally, we see $(\tilde{f},\tilde{\Gamma})$ satisfies \eqref{elliptic} thanks to 
 	\begin{align*}
 		\frac{\partial \tilde{f}}{\partial \mu_i^\lambda} 
 		= 
 		\frac{1}{\varrho(n-\varrho)} \sum_{j=1}^n \frac{\partial f}{\partial\lambda_j}-\frac{1}{\varrho} \frac{\partial f}{\partial\lambda_i}  
 		> 
 		\frac{1}{\varrho \varrho_{\Gamma_{\mathcal{G}}^f}} \big(\sum_{j=1}^n \frac{\partial f}{\partial\lambda_j}- \varrho_{\Gamma_{\mathcal{G}}^f} \frac{\partial f}{\partial\lambda_i}\big)\geq0.
 	\end{align*}
 	The  above two inequalities follow  from $0< n-  \varrho<\varrho_{{\Gamma}_{\mathcal{G}}^f}$ and 
 	Lemma \ref{lemma2-key}.
 	
 \end{proof}

 \begin{theorem}
 	\label{thm1-approx}
 	Given a type 1 pair $(f,\Gamma)$ with $(0,\cdots,0,1)\in \partial{\Gamma}_{\mathcal{G}}^f$, let   $(\tilde{f}, \tilde{\Gamma})$ be the pair as  asserted in Lemma \ref{proposition2-Gamma}. 
 	For $t\geq 1$, denote
 	\begin{equation}
 		\label{pair1-t}
 		\begin{aligned}
 			{f}^t(\lambda)=\tilde{f}
 			\big(t\sum_{i=1}^n \lambda_i\vec{\bf1}-\varrho_{\tilde{\Gamma}_{\mathcal{G}}^{\tilde{f}}} \lambda\big), \,\,
 			{\Gamma}^t=
 			\Big\{\lambda\in\mathbb{R}^n: t\sum_{i=1}^n \lambda_i\vec{\bf1}-\varrho_{\tilde{\Gamma}_{\mathcal{G}}^{\tilde{f}}} \lambda \in\tilde{\Gamma} \Big\}.
 		\end{aligned}
 	\end{equation}
 	Then we have 
 	\begin{enumerate}
 		\item  $\Gamma\subseteq\Gamma^t$ and $\Gamma_{\mathcal{G}}^f\subseteq {\Gamma^t}_{\mathcal{G}}^{f^t}$ for $t>1$.  When $t=1$, $({f}^t,{\Gamma}^t)$ coincides with $({f},\Gamma)$.

 		\item  For each $t>1$, $(f^t,\Gamma^t)$ is a type 2 pair, i.e., $(0,\cdots,0,1)\in \mathring{\Gamma^t}_{\mathcal{G}}^{f^t}$.
 		
 		\item ${f}^t(\lambda)\geq {f}(\lambda), \, \forall \lambda\in {\Gamma}$, $\forall t\geq 1$.
 		
 	\end{enumerate}
 	
 \end{theorem}

 From this,
 each equation of the form  \eqref{equ1-real}  
 can be reformulated as 
 in the following form
 \begin{equation}
 	\label{equation-n-varrho}
 	\begin{aligned}
 		\tilde{f}(\lambda[ \Delta u \cdot g-\varrho\nabla^2u +\tilde{\chi}])=\psi,  
 		\,\,
 		\varrho\neq0,   \nonumber
 	\end{aligned}
 \end{equation}
 and vice versa.

 \begin{proposition} 
 	\label{thm2-equivalent-equations}
 	
 	Fix a pair $(f,\Gamma)$ and pick a positive constant $\varrho$ with
 	$n-\varrho_{{\Gamma}_{\mathcal{G}}^f} \leq \varrho <n$. Let $(\tilde{f},\tilde{\Gamma})$ be 
 	as above
 	such that  $\varrho\leq\varrho_{\tilde{\Gamma}_{\mathcal{G}}^{\tilde{f}}}.$ 
 	Given a  symmetric $(0,2)$-tensor $\chi$, denote
 	$\tilde{\chi}:=(\mathrm{tr}_g\chi)\cdot g-\varrho \, \chi.$
 	Then for any $C^2$-function $u$ with $\lambda[\nabla^2u+\chi]\in\Gamma$,  
 	\begin{equation}
 		\begin{aligned}
 			{f}(\lambda[ \nabla^2u+\chi ])
 			=	\tilde{f}(\lambda[ \Delta u \cdot g-\varrho\nabla^2u +\tilde{\chi}]). \nonumber
 		\end{aligned}
 	\end{equation}	
 	Moreover,  $(0,\cdots,0,1)\in\mathring{{\Gamma}}_{\mathcal{G}}^{{f}} \iff  \varrho<\varrho_{\tilde{\Gamma}_{\mathcal{G}}^{\tilde{f}}}.$
 	
 \end{proposition}

 \section{On the type 2 equations}
 \label{sec3-proof-of-type2-existence}

 \subsection{$C^0$-estimate and construction of   admissible functions}
 \label{subsec1-admissible-C0}

We use Morse theory  to construct admissible functions. 
 The strategy follows   some ideas of   \cite{yuan-PUE-conformal} 
 (see also \cite{yuan-PUE2-note,yuan-PUE4}).
 For our purpose, we need the following:
 \begin{lemma}
 	\label{lemma-diff-topologuy}
 	Let 
 	$\overline{M}$
 	be a compact connected  
 	manifold of real dimension $n\geq 2$ with smooth boundary. Then there is a smooth function $v$ without any critical points. 
 \end{lemma}
 
 \begin{proof}

 	We include a proof, following \cite{yuan-PUE-conformal}.
  	Let $X$ be the double of $\overline{M}$. It is 
  	obtained by 		 gluing two copies of $\overline{M}$ together along their common boundary.
 	 Let $w$ be a smooth Morse function on $X$ with the critical set $\{p_i\}_{i=1}^{m+k}$, among which $p_1,\cdots, p_m$ are all the critical points  being in $\overline{M}$. 
 	Pick $q_1, \cdots, q_m\in X\setminus \overline{M}$ but not the critical point of $w$. By homogeneity lemma 
 	(see \cite{Milnor-1997}), 
 	one can find a diffeomorphism
 	$h: X\to X$, which is smoothly isotopic to the identity, such that  
 	$h(p_i)=q_i$ for $1\leq  i\leq m$, and  	$h(p_i)=p_i$ for $m+1\leq  i\leq  m+k$. 
 	Then $v=w\circ h^{-1}\big|_{\overline{M}}$ is the desired 
 	function.
 	
 \end{proof}

 By Lemma \ref{lemma-diff-topologuy}, 
 we can take a smooth  function  $v$  with  $v\geq1$ and
 $\mathrm{d} v\neq 0$ on $\overline{M}$. Let  
 \begin{equation}
 	\label{const-admissible2}
 	\begin{aligned}
 		\underline{v}=e^{tv}.
 	\end{aligned}
 \end{equation}
 The straightforward computation yields that 
 \begin{align*}
 	\chi+\nabla^2 \underline{v} =\chi +te^{tv} (\nabla^2 v+t\mathrm{d}v\otimes \mathrm{d}v). 
 \end{align*}
 
 Let  $R_0$ be a positive constant such that 
 $
 f(R_0\vec{\bf1})>\psi  \textrm{ in } 
 \overline{M}. 
 $
 By $(0,\cdots,0,1)\in\mathring{{\Gamma}}_{\mathcal{G}}^{{f}}$, we may pick $t\gg1$ such that
 \begin{equation}
 	\begin{aligned}
 		\lambda[\chi +te^{tv} ( \nabla^2 v+\frac{t}{2}\dd v\otimes \dd v)]\in\mathring{{\Gamma}}_{\mathcal{G}}^{{f}},  
 		\,\,
 		\lambda[\frac{t^2 e^{tv}}{2} \dd v\otimes \dd v -R_0 g] \in\mathring{{\Gamma}}_{\mathcal{G}}^{{f}}.
 		\nonumber
 	\end{aligned}
 \end{equation} 
 In particular 
 \begin{equation}
 	\label{admissible-construct2}
 	\begin{aligned}
 		\lambda[ \nabla^2 \underline{v}+\chi]\in \mathring{\Gamma}_{\mathcal{G}}^f \textrm{ in }\overline{M}.
 	\end{aligned}
 \end{equation} 
 
 Using Lemma \ref{lemma2-key} twice,
  we conclude that 
 \begin{equation}
 	\label{ineq2-admissible-construct}
 	\begin{aligned}
 		f( \lambda[	\nabla^2 \underline{v}+\chi]) 
 		\geq f(\lambda[\frac{t^2 e^{tv}}{2} \dd v\otimes \dd v])
 		\geq f(R_0\vec{\bf1})
 		>\psi.
 	\end{aligned}
 \end{equation} 
 The comparison principle then implies that
 \begin{equation}
 	\label{lower-bound1}
 	\inf_{\overline{M}} (u-\underline{v})=
 	\inf_{\partial M} (\varphi-\underline{v}).
 \end{equation}
 
 The upper bound of $u$ is standard. 
 Let $h$ be the solution to 
 \begin{equation}
 	\begin{aligned}
 		\Delta h +\mathrm{tr}_g \chi=0 \mbox{ in } \overline{M}, \,\, \,
 		h=	\varphi \mbox{ on } \partial M. \nonumber
 	\end{aligned}
 \end{equation}
 The existence and regularity of $h$ can be found in standard  textbooks; see \cite[Theorem 6.14]{GT1983}  
 and  \cite[Theorem 4.8]{Aubin1998-nonlinear-book}. 
 Then 
 \begin{equation}
 	\label{upper-bound1}
 	u \leq h \mbox{ in } \overline{M}, \,\, \,
 	u=	h=	\varphi \mbox{ on } \partial M. 
 \end{equation}

 \subsection{Local and global estimates for type 2 equations}

 \begin{proposition}
 	\label{thm1-c1-c2-local-estimates}
 	
 	In addition to \eqref{concave}, \eqref{elliptic} and \eqref{nondegeneracy1}, we assume  
 	$(-\delta_0,0,\cdots,0,1)\in \mathring{\Gamma}_{\mathcal{G}}^f$ for some 	$\delta_0>0$. 
 	Let $B_r\subset M$ be a geodesic ball with radius $r>0$. Let $u\in C^4(M)\cap C^2(\overline{M})$ be an admissible solution to the equation \eqref{equ1-real}. Then 
 	\begin{equation}
 		\label{estimate2-interior}
 		\begin{aligned}
 			\sup_{B_{r/2}} (|\nabla^2u|+|\nabla u|^2)\leq C_1,  
 		\end{aligned}
 	\end{equation}
 	where $C_1$ depends on $\delta_0^{-1}$, $r^{-1}$,  $|u|_{C^0(B_r)}$,  $|\psi|_{C^2(B_r)}$ and geometric quantities 
 	on $B_r$. 
 	Moreover,
 	\begin{equation}
 		\label{estimate2-global}
 		\begin{aligned}
 			\sup_{M}(|\nabla^2u|+|\nabla u|^2)\leq C_2(1+\sup_{\partial M}|\nabla^2u|+\sup_{\partial M}|\nabla u|^2), 
 		\end{aligned}
 	\end{equation}
 	where $C_2$ depends on $\delta_0^{-1}$,  $|u|_{C^0(M)}$, $|\psi|_{C^2(\overline{M})}$ and other known data. Furthermore, both $C_1$ and $C_2$ are independent of $\delta_{\psi,f}^{-1}$, where and hereafter
 	\begin{equation}
 		\label{def1-delta-psi-f}
 		\begin{aligned}
 			\delta_{\psi,f}:=\inf_M\psi-\sup_{\partial\Gamma} f. 
 		\end{aligned}
 	\end{equation}

 \end{proposition}
 
 \begin{proof}
 	By Lemma \ref{lemma1-PUE},  
 	${f}$ is of fully uniform ellipticity, i.e. $f$ satisfies \eqref{pue-3}.
 	As in \cite{Guan2008IMRN} (see also \cite{yuan2020conformal,yuan-PUE-conformal})
 	we can derive local estimates \eqref{estimate2-interior} for second and first order derivatives. 
 	When taking the cutoff function $\eta\equiv1$ in the proof, we get \eqref{estimate2-global}. 
 	Since the proof is somewhat standard, we omit the details here.
 \end{proof}

 \subsection{Boundary estimate for first derivatives}  \label{subsec3-boundary-c1}

 Let $\sigma$ be the distance function 
 to $\partial M$, let $M_{\delta}:=\{z\in M: \sigma(z)<\delta\}$.  
In the following, we use    
 $\sigma$  
 to construct a local   $\mathring{{\Gamma}}_{\mathcal{G}}^{{f}}$-admissible function $\underline{w}$ subject to
 \begin{equation} 	
 	\label{local-subsolution2-inequ2}
 	\begin{aligned}  	
 		f(\lambda[\nabla^2 \underline{w}+\chi])  
 		>\psi \mbox{ in } M_{\delta_0},  
 	\end{aligned}  
 \end{equation} 
 \begin{equation}
 	\label{admifunction1-local}
 	\begin{aligned}
 		u\geq \underline{w}  \mbox{ in }   M_{\delta_0}, \,\,
 		\underline{w}=\varphi \mbox{ on } \partial M, 
 	\end{aligned}
 \end{equation} 
 for  some   $\delta_0>0$. 
 As in   \cite{Guan2008IMRN}   take
 $	w(z)=  
 2\log \frac{\delta^2}{\delta^2 +\sigma(z)}.$ 
 The  computation gives  
 \begin{equation}
 	\begin{aligned}
 		\chi+\nabla^2 (w+\varphi) 
 		= \chi +\nabla^2\varphi+\frac{2\dd \sigma\otimes \dd \sigma}{(\delta^2+ \sigma)^2}-\frac{2\nabla^2\sigma}{\delta^2+\sigma}. \nonumber
 	\end{aligned}
 \end{equation}   
 Notice  $|\nabla\sigma| =1$ on $\partial M$ and $(0,\cdots,0,1)\in\mathring{{\Gamma}}_{\mathcal{G}}^{{f}}$. Take $0<\delta_1\ll1$ so that 
 \begin{equation} 	
 	\begin{aligned}  	
 		\lambda\big[\chi+\nabla^2\varphi+   \frac{\dd\sigma \otimes \dd\sigma}{(\delta^2+ \sigma)^2}-\frac{2\nabla^2\sigma }{\delta^2+ \sigma}\big]
 		\in \mathring{{\Gamma}}_{\mathcal{G}}^{{f}} \textrm{ in } M_{\delta}  \nonumber
 	\end{aligned}  
 \end{equation} 
 for $0<\delta\leq \delta_1$. 
 By Lemma  \ref{lemma2-key},
 there is $0<\delta_2\leq \delta_1$ such that 
 for $0<\delta\leq \delta_2$
 \begin{equation} 	
 	\begin{aligned}  	
 		f(\lambda[\chi+\nabla^2 (w+\varphi)]) \geq \,&
 		f(\lambda[\frac{\dd\sigma\otimes\dd\sigma}{(\delta^2+ \sigma)^2}])  
 		>\psi \mbox{ in } M_{\delta}. \nonumber 
 	\end{aligned}  
 \end{equation}  
 From \eqref{lower-bound1} there is a  positive uniform constant $\delta_3$ such that  
 \begin{equation} 
 	\label{lower-bound2}	
 	\begin{aligned} 
 		\underset{M}\inf(u-\varphi)\geq  
 		2\log\frac{\delta_3}{1+\delta_3}.  	 \nonumber
 \end{aligned} \end{equation}

 Denote $\delta_0=\min\{\delta_2,\delta_3\},$  and take
 \begin{equation}
 	\label{barrier1-w}
 	\begin{aligned}	
 		\underline{w}=2\log \frac{\delta_0^2}{\delta_0^2 +\sigma}+\varphi.  
 	\end{aligned}
 \end{equation} 
 Consequently, the comparison principle 
 yields that 
 \begin{equation}
 	\label{lower-comparison}
 	\begin{aligned}
 		\inf_{M_{\delta_0}}(u - \underline{w}) 
 		= \inf_{\partial M_{\delta_0}}		(u-\underline{w}) 
 		\geq 0. \nonumber
 	\end{aligned}
 \end{equation}
 Thus $\underline{w}$ satisfies \eqref{admifunction1-local} and \eqref{local-subsolution2-inequ2}    in $M_{\delta_0}.$
 From  \eqref{upper-bound1} and \eqref{admifunction1-local}, we get
 \begin{equation}
 	\label{boundary-gradient-1}	
 	\begin{aligned}  	
 		|\nabla u|\leq C \mbox{ on }  \partial  M.
 	\end{aligned} 
 \end{equation}
 
 \subsection{Boundary estimate for second derivatives}

 Now we   review some notation.  
 For   $x_0\in\partial M$ we choose a smooth  local orthonormal frame $e_1, \cdots, e_n$ around $x_0$ such that
 $e_n$ is unit inner normal vector field when restricted to $\partial M$. 
 Denote
 $ 
 g_{ij} 
 =g(e_i,e_j)$ and  $\{g^{ij} \} =  \{g_{ij} \}^{-1}.$
 Under the Levi-Civita connection  of $g$, $\nabla_{e_i}e_j=\Gamma_{ij}^k e_k$, and $\Gamma_{ij}^k$ denote the Christoffel symbols.  
 For simplicity we write 
 $$\nabla_i=\nabla_{e_i}, \nabla_{ij}=\nabla_i\nabla_j-\Gamma_{ij}^k\nabla_k, 
 \nabla_{ijk}=\nabla_i\nabla_{jk}-\Gamma_{ij}^l\nabla_{lk}-\Gamma^l_{ik}\nabla_{jl}, \cdots,\mbox{ etc}.$$ 
 The linearized operator of   \eqref{equ1-real} at the solution $u$, 
 say $\mathcal{L}$,
 is  locally 
 given by
 \begin{equation}
 	\label{def1-linearized-operator}
 	\mathcal{L} w=F^{i j} \nabla_{i j} w,  
 \end{equation}
 where 
 $\mathfrak{g}_{i j}=\chi_{i  j}+\nabla_{i j} u$ and  $F^{i j}=\frac{\partial F(\mathfrak{g})}{\partial \mathfrak{g}_{i j}}$.

 Let $\underline{w}$ be as in \eqref{barrier1-w} a local $\mathring{{\Gamma}}_{\mathcal{G}}^{{f}}$-admissible function obeying \eqref{admifunction1-local} and \eqref{local-subsolution2-inequ2}.  
 Applying
 such a local barrier function 
 $\underline{w}$ to Lemma \ref{lemma5-key},  there is a positive constant  $\varepsilon$ (independent of $\delta_{\psi,f}$) 
 such that   either  
 \begin{equation}
 	\label{inequ-case1-0}
 	\begin{aligned}
 		F^{i j} (\nabla_{ij}\underline{w} -\nabla_{ij}u )
 		\geq  \varepsilon F^{i j}g_{i  j} +\varepsilon,  
 	\end{aligned}
 \end{equation}
 or
 \begin{equation}
 	\label{inequ-case2-0}
 	\begin{aligned}
 		F^{i j}\geq  \varepsilon (1+F^{p  q}g_{p  q})g^{i j}.
 	\end{aligned}
 \end{equation}
 
 \begin{proposition}
 	\label{thm2-boundary-c2-estimate} 
 	Let $\delta_{\psi,f}$ be as in \eqref{def1-delta-psi-f}.
 	In addition to \eqref{concave}, \eqref{elliptic} and  \eqref{nondegeneracy1}, we assume    $(-\delta_0,0,\cdots,0,1)\in \mathring{\Gamma}_{\mathcal{G}}^f$ for some 	$\delta_0>0$.
 	Let $u\in C^4(M)\cap C^2(\overline{M})$ be an admissible solution to the Dirichlet problem  \eqref{equ1-real}-\eqref{boundary-data1}. Then
 	\begin{equation}
 		\label{estimate1-boundary-second}	\begin{aligned}
 			\sup_{\partial M}|\nabla^2u|\leq C, \textrm{ independent of } \delta_{\psi,f}^{-1},
 		\end{aligned}
 	\end{equation}
 	where $C$ is a  positive constant depending only on  $\delta_0^{-1}$, $|u|_{C^0(M)}$, $|\nabla u|_{C^0(\partial M)}$, $|\underline{w}|_{C^2(M_{\delta_0})}$,  
 	$|\varphi|_{C^3(\partial M)}$, and other known data. 
 \end{proposition}
 
 \begin{proof}

 	The proof of the boundary estimate for pure tangential derivatives
 	\begin{equation}  	
 		\label{bdry-estimate1}	
 		\begin{aligned}
 			|\nabla_{\alpha \beta}u(x_0)|\leq C_5, \, \forall \alpha, \beta\leq  n-1
 		\end{aligned}
 	\end{equation}
 	is standard by the boundary value condition. 
 	Here $C_5$ depends only on $|\varphi|_{C^2(\partial M)}$, $|u|_{C^1(\partial M)}$ and geometric quantities of $\partial M$.
 	
 	By a minor modification of 
 	\cite[Proposition 3.1]{yuan-PAMS2025},
 	we can employ \eqref{inequ-case1-0} and \eqref{inequ-case2-0}
 	to derive boundary estimate for mixed derivatives 
 	\begin{equation}  		\label{bdry-estimate2}		
 		\begin{aligned}
 			|\nabla_{\alpha n}u(x_0)|\leq C_6 (1+\sup_M |\nabla u|), \, \forall 1\leq \alpha< n,   \textrm{ independent of } \delta_{\psi,f}^{-1},
 		\end{aligned}
 	\end{equation}
 	where $C_6$ depends on $|u|_{C^0(\overline{M})}$, $|\varphi|_{C^3(\partial M)}$, $|\psi|_{C^1(\overline{M})}$, 
 	and other known data.

 	Let $\underline{v}$ be as in \eqref{const-admissible2}. 
 	From \eqref{concave-1} and 
 	\eqref{ineq2-admissible-construct},
 	we get
 	\begin{align*}
 		F^{n  n} \nabla_{n  n}u(x_0)\leq -\sum_{i+j<2n}  F^{i j}\nabla_{i  j} u(x_0)+\sum_{i,j=1}^nF^{i  j} \nabla_{i j} \underline{v}(x_0).
 	\end{align*}
 	By Lemma \ref{lemma1-PUE} or \eqref{pue-3}, 
 	$F^{n n}\geq 
 	\frac{\delta_0}{n} \sum_{i=1}^n F^{i i}$ at $x_0$. Thus   
 	\begin{equation}
 		\label{dounble-normal2}
 		\begin{aligned}
 			u_{n n}(x_0)\leq \frac{n}{\delta_0}  \sum_{i+j<2n}   |\nabla_{i j} u(x_0)|+ \frac{n}{\delta_0}\sum_{i,j=1}^n |\nabla_{i j} \underline{v} (x_0)|. 
 		\end{aligned}
 	\end{equation} 
 	Finally, by  \eqref{bdry-estimate1}   
 	and $\Gamma\subseteq\Gamma_1$, we see $\nabla_{nn}u(x_0)\geq -C$.  
 \end{proof}

 \subsection{Completion of the proof of Theorem \ref{thm1-type2}} 
 
 Let $\underline{v}$ be a  smooth   function,
  as constructed in \eqref{const-admissible2}, 
  satisfying    \eqref{admissible-construct2} and \eqref{ineq2-admissible-construct}. 
 For each $t\in [0,1]$, we consider 
 \begin{equation}
 	\label{equation-t}
 	\left\{ 
 	\begin{aligned}
 		f(\lambda[\nabla^2u^t+\chi])=\,& t\psi+(1-t) f(\lambda[\nabla^2 \underline{v}+\chi])  \,& \textrm{ in } \overline{M}, \\
 		u^t=\,& t\varphi+(1-t)\underline{v}  \,& \textrm{ on } \partial M. 
 	\end{aligned}
 	\right.
 \end{equation}  
 
 Denote
 \begin{equation}
 	\begin{aligned}
 		\mathrm{S}=\big\{t\in [0,1]:  
 		\textrm{\eqref{equation-t} has a unique smooth admissible solution for } t \big\}. \nonumber
 	\end{aligned}
 \end{equation}
 Obviously, for $t=0$ the Dirichlet problem has a unique admissible solution $\underline{v}$. 
 By the implicit function theorem and  
 the theory of linear elliptic equations, $\mathrm{S}$ is an  open subset of $[0,1]$.  
 Together with the Evans-Krylov theorem and Schauder theory, the  \eqref{lower-bound1},  \eqref{upper-bound1},  \eqref{estimate2-global}, \eqref{boundary-gradient-1} and  \eqref{estimate1-boundary-second} imply that  $\mathrm{S}$ is closed. 
 Thus $\mathrm{S}=[0,1]$.

 We   remark that subsolution assumption 
 is not needed in the present proof.

 \subsection{Proof of Theorem \ref{thm3-type2-degenerate}} 
 
 By  Theorem  \ref{thm1-type2}, 
 for each $\delta>0$ there is a unique smooth admissible function $u^\delta$ satisfying
 \begin{equation}
 	\label{equation2-t1}
 	\begin{aligned} 
 		f(\lambda[\nabla^2 u^\delta+\chi]) =\psi+\delta \textrm{ in } \overline{M}, \,\,
 		u^\delta =\varphi \textrm{ on } \partial M. 
 	\end{aligned}
 \end{equation}
 From  \eqref{lower-bound1},  \eqref{upper-bound1},  \eqref{estimate2-global}, \eqref{boundary-gradient-1} and \eqref{estimate1-boundary-second}, 
 we deduce that 
 \begin{equation}
 	\label{real-hessian2-estimate}
 	\begin{aligned}
 		\sup_{\overline{M}}|\nabla^2 u^\delta|\leq C, \textrm{   independent of } \delta^{-1}.
 	\end{aligned}
 \end{equation} 
 Take $\delta\to 0^+$, by passing to a proper subsequence, we may obtain a weak 
 $C^{1,1}$-solution 
 to Dirichlet problem for degenerate equation.

\section{Quantitative second order estimate}
\label{sec4}

We establish a real analogue of \cite[Theorem 1.7]{yuan1-closed} in the unbounded case.

\begin{proposition} 
\label{prop2-1-real}   
In addition to \eqref{concave}, \eqref{elliptic}, \eqref{unbounded-1}  and \eqref{nondegeneracy1}, 
we assume that the Dirichlet problem 
\eqref{equ1-real}-\eqref{boundary-data1} admits an admissible subsolution $\underline{u}\in C^2(\overline{M})$. 
For any  admissible solution  $u\in C^4(M)\cap C^2(\overline{M})$ to the
Dirichlet problem,  we have  second estimate of the quantitative form
\begin{equation}
	\label{boundary-estimate0-quantitative}
	\begin{aligned}
		\sup_{\overline{M}}|\nabla^2 u| \leq C(1+\sup_{  M}|\nabla u|^2),
	\end{aligned}
\end{equation}
and  the Hessian estimate \eqref{estimate1-Hessian-real}. 
\end{proposition}

We first summarize some notation.  
Fix a point $x_0\in\partial M$, we choose a smooth  local orthonormal frame $e_1, \cdots, e_n$ around $x_0$ such that
$e_n$ is unit inner normal vector field when restricted to $\partial M$.  Denote $$\Gamma_\infty=\{(\lambda_{1}, \cdots, \lambda_{n-1}): (\lambda_{1}, \cdots, \lambda_{n-1},R)\in\Gamma \mbox{ for some } R>0\}.$$ 
Throughout this section,  the Greek letters $\alpha, \beta$ run from $1$ to $n-1$.

The main issue is to derive \eqref{boundary-estimate0-quantitative}.
Following an estimate presented in
\cite[Section 3]{Guan-Dirichlet},
we can prove that
\begin{equation}
\begin{aligned}
	\sup_{\overline{M}}|\nabla^2 u| \leq C(1+\sup_{M}|\nabla u|^2+\sup_{\partial M}|\nabla^2  u|). \nonumber
\end{aligned}
\end{equation}
In order to derive \eqref{boundary-estimate0-quantitative}, it only requires to prove the quantitative boundary estimate 
\begin{equation}
\label{boundary-estimate1-quantitative}
\begin{aligned}
	\sup_{\partial M}|\nabla^2 u| \leq C(1+\sup_{  M}|\nabla u|^2).
\end{aligned}
\end{equation}

From \cite[Proposition 3.1]{yuan-PAMS2025},
we have boundary estimate for mixed derivatives
\begin{equation}  		\label{bdry-estimate2-2}		
\begin{aligned}
	|\nabla_{\alpha n}(x_0)|\leq C  (1+\sup_M|\nabla u|), \, \forall 1\leq \alpha\leq n-1.  
\end{aligned}
\end{equation}
By 
\eqref{bdry-estimate1} and $\Gamma\subseteq\Gamma_1$, 
we see $\nabla_{nn}u(x_0)\geq -C$.  
So,
it suffices to prove the following lemma. This is a real analogue of \cite[Proposition 7.1]{yuan1-closed}.
\begin{lemma}
\label{lemma2-double-normal}
In the setting of Proposition \ref{prop2-1-real}, we have 
\begin{equation}	\label{boundary-estimate2-quantitative}
	\begin{aligned}
		\mathfrak{g}_{nn}(x_0) \leq C\Big(1+\sum_{\alpha=1}^{n-1} |\mathfrak{g}_{\alpha n}(x_0)|^2\Big). 
	\end{aligned}
\end{equation}
\end{lemma}

Let $\sigma$ 
be as in Subsection \ref{subsec3-boundary-c1}  the distance to
the boundary. 
At  $x_0\in\partial M$,
the boundary value condition yields  
\begin{equation}
\label{410-buchong}
\begin{aligned}
	\mathfrak{g}_{\alpha \beta}= (1-t)\underline{\mathfrak{g}}_{\alpha \beta}
	+\{t\underline{\mathfrak{g}}_{\alpha \beta}+ \nabla_n(u-\underline{u})  \nabla_{\alpha \beta} \sigma \}.
\end{aligned}
\end{equation}  
Let
$t_0$ be the first  
$t$ as we decrease $t$ from $1$  
such that
\begin{equation}
\label{key0-yuan3}
\begin{aligned} 
	\lambda(t\underline{\mathfrak{g}}_{\alpha \beta}+ \nabla_n(u-\underline{u})  \nabla_{\alpha \beta} \sigma)\in\partial\Gamma_\infty 
\end{aligned}
\end{equation}
at $x_0$,
where $\partial\Gamma_\infty$ is the boundary of $\Gamma_\infty$. 
Obviously,
\begin{equation}
\label{1-yuan3}
\begin{aligned}
	-T_0< t_0<1 
\end{aligned}
\end{equation}
for some uniform positive constant $T_0$ under control. 
Following closely the outline of \cite{CNS3} (see also \cite{LiSY2004} in the complex case),  
one may obtain  the following 
\begin{equation}
\label{bound-t0}
\begin{aligned}
	(1-t_0)^{-1}\leq  C, \textrm{ possibly depending on } \delta_{\psi,f}^{-1},
\end{aligned}
\end{equation}
where $C$  is a uniform positive constant  depending on 
$|u|_{C^0(M)}$, $|\nabla u|_{C^0(\partial M)}$, $|\underline{u}|_{C^2(M)}$, $\sup_M\psi$, 
$\partial M$ up to third derivatives and other known data.  

Denote  
\begin{equation} 	\label{yuan3-buchong6} 
\begin{aligned} 	
	\underline{\lambda}'=(\underline{\lambda}'_1,\cdots,\underline{\lambda}'_{n-1})	\equiv  \lambda( \underline{\mathfrak{g}}_{\alpha \beta}),
\end{aligned}
\end{equation} 
\begin{equation}
\label{yuan3-buchong3} 	\begin{aligned}  
	(A_{t_0})_{\alpha  \beta}= t_0\underline{\mathfrak{g}}_{\alpha \beta}+ \nabla_n(u-\underline{u})  \nabla_{\alpha \beta} \sigma, \,\,
	\tilde{\lambda}'=(\tilde{\lambda}_1',\cdots,\tilde{\lambda}_{n-1}') \equiv 	
	\lambda [(A_{t_0})_{\alpha  \beta}]. 	\end{aligned}
\end{equation} 
For $R>0$,  define the matrix  
\begin{equation}
\label{def1-AR}
\begin{aligned}
	{A}(R) =
	\begin{pmatrix}
		\mathfrak{{g}}_{\alpha \beta} &\mathfrak{g}_{\alpha  n}\\
		\mathfrak{g}_{n \beta}& R  
	\end{pmatrix}.
\end{aligned}
\end{equation}

By 
\eqref{410-buchong} we can decompose $A(R)$ as follows
\begin{equation}
\begin{aligned}
	{A}(R) = A'(R)+A''(R) 
\end{aligned}
\end{equation}
where 
\[A'(R)=\begin{pmatrix}
(1-t_0)(\mathfrak{\underline{g}}_{\alpha  \beta}-\frac{\varepsilon_0}{4}
\delta_{\alpha\beta}) &\mathfrak{g}_{\alpha  n}\\
\mathfrak{g}_{n \beta}& R/2  \nonumber
\end{pmatrix}, \mbox{  }
A''(R)=\begin{pmatrix} (A_{t_0})_{\alpha  \beta}+\frac{(1-t_0)\varepsilon_0}{4} \delta_{\alpha\beta} &0\\ 0& R/2  \end{pmatrix}.\]

Moreover, we 
need the following  lemma, which was proved by \cite{yuan1-closed} in complex case.
\begin{lemma}
\label{lemma1-bdyestimate}
There exist uniform constants $R_3$ and $C_{A''}$ depending on
$(1-t_0)^{-1}$, $\varepsilon_0^{-1}$,   $f$ and other known data under control, such that 
\begin{equation}
	\label{at0-2}
	\begin{aligned}
		f(2\lambda[A''(R_3)] )\geq  -C_{A''}.
	\end{aligned}
\end{equation}

\end{lemma}

\begin{proof}

We include the proof, following \cite{yuan1-closed}.
Let     	
$\underline{\lambda}'$ and 
$\tilde{\lambda}'$ be as in \eqref{yuan3-buchong6} and \eqref{yuan3-buchong3}, respectively.
One can see that 
there is a uniform constant $C_0>0$ depending on $|t_0|$,
$\sup_{\partial M}|\nabla u|$ and other known data, such that $|\tilde{\lambda}'|\leq   C_0$, that is
\begin{equation}
	\label{yuan3-buchong4v}
	\begin{aligned}
		\tilde{\lambda}'\in {K}\equiv \{\lambda'\in \overline{\Gamma}_\infty: |\lambda'|\leq   C_0\}.  \nonumber
	\end{aligned}
\end{equation}
There is a 
positive constant $R_2$ depending on $((1-t_0)\varepsilon_0)^{-1}$, $K$ and other known data, 
such that  
\begin{equation}
	\label{key-74}
	\begin{aligned}
		\lambda	\Big[\begin{pmatrix} (A_{t_0})_{\alpha  \beta}+\frac{(1-t_0)\varepsilon_0}{8} \delta_{\alpha\beta} &0\\ 0& R_2/2  \end{pmatrix}\Big] \in \Gamma.
	\end{aligned}
\end{equation}
Let $R_3=R_2+ \frac{(1-t_0)\varepsilon_0}{4}$, we have
\begin{equation}
	\label{at0-1}
	\begin{aligned}
		\lambda[A''(R_3)] \in \Gamma+\frac{(1-t_0)\varepsilon_0}{8} \vec{\bf1}.
	\end{aligned}
\end{equation} 
Notice that $A''(R_3)$ is bounded.
There exists a uniform constant $C_{A''}$ depending on  $(1-t_0)^{-1}$, $\varepsilon_0^{-1}$, $R_3$  and $f$, such that
\eqref{at0-2} holds.

\end{proof}

\begin{proof}
[Proof of Lemma \ref{lemma2-double-normal}]
We follow closely \cite{yuan2017,yuan2019,yuan1-closed}.
By \eqref{unbounded-1} and \eqref{bound-t0},
there is a uniform positive constant $R_1$  
such that  	 
\begin{equation}
	\label{key-03-yuan3}
	\begin{aligned} 	 f\left(2(1-t_0)(\underline{\lambda}'_{1}-{\varepsilon_0}/2),\cdots, 2(1-t_0)(\underline{\lambda}'_{n-1}
		-{\varepsilon_0}/2), 2R_1\right) \geq   2f(\underline{\lambda})+ {C_{A''}}, 
	\end{aligned}
\end{equation}
and  $(\underline{\lambda}'_1-\varepsilon_{0}/2, \cdots, \underline{\lambda}'_{n-1}-\varepsilon_0/2, {(1-t_0)^{-1}}{R_1})\in \Gamma.$

Let's pick  $\epsilon=\frac{(1-t_0)\varepsilon_0}{4}$ in 
Lemma  \ref{yuan's-quantitative-lemma}. And then 
we set
\begin{equation}
	\label{def1-R-c}
	\begin{aligned}
		R_c\equiv  \,&
		\frac{8(2n-3)}{(1-t_0)\varepsilon_0}\sum_{\alpha=1}^{n-1} | \mathfrak{g}_{\alpha  n}|^2 	 	  
		+ 2(n-1) (1-t_0) \sum_{\alpha=1}^{n-1}|{\underline{\lambda}_\alpha'}|
		 \\\,&	
		+\frac{n(n-1)(1-t_0)\varepsilon_0}{2}
		+2R_1+2R_2,   \nonumber 
	\end{aligned}
\end{equation}
where $\varepsilon_0$, $R_1$ and $R_2$ are the constants we 
fixed in \eqref{key-03-yuan3} and \eqref{at0-1}.
According to Lemma \ref{yuan's-quantitative-lemma},
the eigenvalues $\lambda[{A}'(R_c)]$ 
(possibly with an appropriate order)
shall behave like
\begin{equation}
	\label{lemma12-yuan}
	\begin{aligned}
		\lambda_{\alpha}[{A}'(R_c)]\geq   (1-t_0)  (\underline{\lambda}'_{\alpha}-{\varepsilon_0}/{2}),
		 \,  1\leq   \alpha\leq   n-1;
		\,\, 
			\lambda_{n}[{A}'(R_c)]\geq  R_c/2.
	\end{aligned}
\end{equation} 	 
In particular, $\lambda[{A}'(R_c)]\in \Gamma$. So $\lambda[A(R_c)]\in \Gamma$.  

Using the concavity of $f$,  Lemma  \ref{lemma1-bdyestimate}
and \eqref{key-03-yuan3},  
\eqref{lemma12-yuan}, we deduce that
\begin{equation}
	\label{yuan-k1}
	\begin{aligned}
		f(\lambda[A(R_c)])\geq  
		\frac{1}{2}f(2\lambda[A'(R_c)])+\frac{1}{2}f(2\lambda[A''(R_c)])  
		\geq  f(\underline{\lambda})\geq  \psi. \nonumber
	\end{aligned}
\end{equation} 
Thus $\mathfrak{g}_{n n} \leq  R_c.$   
By \eqref{bound-t0}, $(1-t_0)^{-1}$ is bounded.
The proof is now complete.

\end{proof}

\begin{proof}
[Proof  of Proposition \ref{prop2-1-real}]

By \eqref{upper-bound1} and  subsolution assumption, we get
\begin{align*}
	\sup_M|u|+\sup_{\partial M} |\nabla u|\leq C.
\end{align*} 
Based on \eqref{boundary-estimate0-quantitative}, without using the assumption \eqref{addistruc},
we  use the blow-up argument   
 developed in \cite{yuan1-closed}, 
extending that of  \cite{Dinew2017Kolo,Gabor} to  optimal setting, 
to derive   gradient estimate 
\begin{align*}
	\sup_M|\nabla u|\leq C.
\end{align*}  
(The details can also be found in \cite[Section 5]{yuan-PAMS2025}.) 
In conclusion, we get \eqref{estimate1-Hessian-real}.
\end{proof}

\section{A continuity method and proof of Theorem  \ref{thm1-exietnce-Dirichlet-real}  }
\label{sec2-contunuitymethod-Dirichlet}

We now give a continuity method for  Dirichlet problem \eqref{equ1-real}-\eqref{boundary-data1}. 
Our method is based on Theorems  \ref{thm1-type2} and \ref{thm1-approx}.
Let $(f,\Gamma)$ be the pair corresponding to  \eqref{equ1-real}. 
If $(0,\cdots,0,1)\in \mathring{\Gamma}_{\mathcal{G}}^f$, then we are done as in Theorem  \ref{thm1-type2}.
From now on, we assume $(0,\cdots,0,1)\in \partial{\Gamma}_{\mathcal{G}}^f$ and then take  the pair  $(f^t,\Gamma^t)$ defined as in \eqref{pair1-t}. 
For each $t\in [1,2]$, consider  
\begin{equation}
\label{equation2-t}
\begin{aligned} 
	{f}^t (\lambda[\nabla^2 u+\chi]) =\psi \textrm{ in } \overline{M}, \,\,
	u=\varphi \textrm{ on } \partial M. 
\end{aligned}
\end{equation}

When $t>1$, $(f^t,\Gamma^t)$ is of type 2 according to Theorem \ref{thm1-approx}. Then by Theorem  \ref{thm1-type2}, for each $t>1$
\eqref{equation2-t} has a unique smooth solution $u^t$ with  
$\lambda[\nabla^2u^t+\chi]\in {\Gamma}^t$.

Let $\underline{u}$  be an admissible subsolution to 
the Dirichlet problem 
\eqref{equ1-real}-\eqref{boundary-data1}. 
By (3) and (1) of Theorem \ref{thm1-approx},  $\underline{u}$ is also an admissible subsolution to \eqref{equation2-t} for all $1\leq t\leq 2$
\begin{equation}
\begin{aligned}
	{f}^t(\lambda[\nabla^2\underline{u}+\chi])
	\geq \psi, 
	\,\, \lambda[\nabla^2\underline{u}+\chi]\in
	\Gamma^t  \textrm{ in } \overline{M}. \nonumber
\end{aligned}
\end{equation} 
By \eqref{unbounded-1}, 
Lemma \ref{lemma1-PUE} and Theorem \ref{thm1-approx} ($t>1$), 
$(f^t,\Gamma^t)$    uniformly obeys   \eqref{unbounded-1} for   $t\in[1,2]$. 
By   Proposition \ref{prop2-1-real}, 
for $t\in (1,2]$  there is some $\alpha\in (0,1)$ such that 
\begin{equation}
\begin{aligned}
	|u^t|_{C^{4,\alpha}(\overline{M})} \leq C, \textrm{ independently of } t. 
	\nonumber
\end{aligned}
\end{equation}
Passing to a proper subsequence as $t\to 1^+$, we obtain  a $C^{4,\alpha}$-admissible solution $u$ to the Dirichlet problem \eqref{equ1-real}-\eqref{boundary-data1}. 
The Schauder theory implies that $u$ is indeed smooth. The uniqueness follows from the maximum principle.


\begin{appendix}

\section{Useful lemmas}
\label{appendix1}

The following lemma is a key ingredient  
 for the proof of Proposition \ref{prop2-1-real}.


\begin{lemma}
	[\cite{yuan2017,yuan2019,yuan1-closed}]
	\label{yuan's-quantitative-lemma}
	Let $A$ be an $n\times n$ 
	symmetric matrix
	\begin{equation}\label{matrix3}\left(\begin{matrix}
			d_1&&  &&a_{1}\\ &d_2&& &a_2\\&&\ddots&&\vdots \\ && &  d_{n-1}& a_{n-1}\\
			a_1&  a_2&\cdots&  a_{n-1}& \mathrm{{\bf a}} 
		\end{matrix}\right)\end{equation}
	with $d_1,\cdots, d_{n-1}, a_1,\cdots, a_{n-1}$ fixed, and with $\mathrm{{\bf a}}$ variable.
	Denote the eigenvalues of $A$ by $\lambda=(\lambda_1,\cdots, \lambda_n)$.
	Let $\epsilon>0$ be a fixed constant.
	Suppose that  the parameter $\mathrm{{\bf a}}$ in $A$ satisfies  the quadratic
	growth condition  
	\begin{equation}
		\begin{aligned}
			\label{guanjian1-yuan}
			\mathrm{{\bf a}}\geq \frac{2n-3}{\epsilon}\sum_{i=1}^{n-1}|a_i|^2 +(n-1)\sum_{i=1}^{n-1} |d_i|+ \frac{(n-2)\epsilon}{2n-3}.
		\end{aligned}
	\end{equation}
	Then the eigenvalues (possibly with a proper permutation) behave like
	\begin{equation}
		\begin{aligned}
			|d_{\alpha}-\lambda_{\alpha}|
			<   \epsilon, \mbox{ } \forall 1\leq \alpha\leq n-1;\,\,
			0\leq \lambda_{n}-\mathrm{{\bf a}}
			< (n-1)\epsilon. \nonumber
		\end{aligned}
	\end{equation}
\end{lemma}

The above lemma was  proposed 
in  \cite{yuan2017,yuan2019}.
We also refer the interested reader to  
\cite{yuan1-closed}
for a detailed proof. 

Below we also present a lemma regarding subsolution. 
Denote $\partial\Gamma^\sigma:=\{\lambda\in\Gamma: f(\lambda)=\sigma\}.$

\begin{lemma}
	[\cite{yuan1-closed}]
	\label{lemma5-key}
	Assume   \eqref{concave} and  \eqref{elliptic} hold.
	Fix $\sigma\in (\sup_{\partial \Gamma}f,\sup_\Gamma f)$. 
	Suppose   $\mu\in\mathbb{R}^{n}$ 
	obeys for some  $\delta$, $R>0$ that
	\begin{equation}
		\label{set2-bound-2}
		\begin{aligned} 
			(\mu-2\delta \vec{\bf 1}+\Gamma_n) \cap \partial\Gamma^\sigma \subset B_R(0). 
		\end{aligned}
	\end{equation} 
	Then there is a positive constant $\varepsilon$ depending on  $\sigma$, $\mu$, $\delta$, $R$ such that if $\lambda\in\partial\Gamma^\sigma$, then either
	\begin{equation}
		\label{inequ-case1-2}
		\begin{aligned}
			\sum_{i=1}^n f_i(\lambda) (\mu_i-\lambda_i) \geq  \varepsilon \big(1+\sum_{i=1}^n f_i(\lambda)\big)   
		\end{aligned}
	\end{equation}
 or 
	\begin{equation}
		\label{inequ-case2-2}
		\begin{aligned}
			f_j(\lambda) \geq  \varepsilon \big(1+\sum_{i=1}^n f_i(\lambda) \big), \, \forall 1\leq  j\leq  n.
		\end{aligned}
	\end{equation}
	
\end{lemma}

\begin{proof}  
	
	For self-sufficiency, we include the details,  following the  proof of \cite{yuan1-closed}.
	Denote $\tilde{\mu}=\mu-\delta \vec{\bf 1}$. By \eqref{set2-bound-2},
	there is a  positive constant $R_0$ depending only on $\tilde{\mu}$ and $R$ such that 
	\begin{equation}
		\label{key1-proof-lemma3.2}
		\begin{aligned}  
			f(\tilde{\mu}+R_0\vec{\bf e}_i)>\sigma, \,\, \tilde{\mu}+R_0\vec{\bf e}_i\in\Gamma, \,\, \forall 1\leq  i\leq  n, 
		\end{aligned}
	\end{equation} 
	where $\vec{\bf e}_i$ 
	is the $i$-th standard basis vector in $\mathbb{R}^n$. 
	Denote $\delta_0=\underset{1\leq  i\leq  n}{\min}\{f(\tilde{\mu}+R_0\vec{\bf e}_i)-\sigma\}$.   
	Below we prove the conclusion   for $\varepsilon= \min\{ \frac{\delta}{2}, \frac{\delta}{2R_0}, \frac{\delta_0}{2}, \frac{\delta_0}{2R_0}\}.$
	Without loss of generality,  
	we assume 
	$\lambda_1\leq  \cdots\leq \lambda_n$. Then $f_1(\lambda)\geq  \cdots \geq  f_n(\lambda).$ 
	We   get \eqref{inequ-case2-2} if
	$f_n(\lambda)>\frac{\delta_0}{2R_0} +\frac{\delta}{2R_0}  \sum_{i=1}^n f_i(\lambda). $
	From now on we assume  $\frac{\delta_0}{2}+	\frac{\delta}{2}  \sum_{i=1}^n f_i(\lambda) \geq  R_0f_n(\lambda). $ 
	Combining this assumption with \eqref{concave-1}  we have
	\begin{equation} 
		\begin{aligned}
			\sum_{i=1}^n f_i(\lambda) (\mu_i-\lambda_i)    
			=\,&   	\sum_{i=1}^n f_i(\lambda) (\tilde{\mu}_i-\lambda_i) 
			+\delta\sum_{i=1}^n f_i(\lambda)   
			\\ 
			\geq \,& 
			f(\tilde{\mu}+R_0\vec{\bf e}_n)-f(\lambda)-R_0f_n(\lambda) +\delta \sum_{i=1}^n f_i(\lambda)
			\\ 
			\geq  \,& 
			\frac{\delta_0}{2}+ \frac{\delta}{2} \sum_{i=1}^n f_i(\lambda).   \nonumber
		\end{aligned}
	\end{equation}  
	
\end{proof}

\begin{remark}	According to \cite[Lemma 3.1]{yuan1-closed}, the right-hand sides of both inequalities \eqref{inequ-case1-2} and \eqref{inequ-case2-2} can be strengthened to $\epsilon\big(1+\sum_{i=1}^n f_i(\lambda)+|\sum_{i=1}^nf_i(\lambda)\lambda_i|\big).$\end{remark}

\end{appendix}

\bibliographystyle{amsplain}



\end{document}